\documentclass[12pt]{amsart}
\usepackage{listings} 
\usepackage{latexsym, amsmath, amssymb,amsthm,amsopn,amsfonts}
\usepackage{version}
\usepackage{epsfig,graphics,color,graphicx,graphpap}
\usepackage{amssymb}
\usepackage{todonotes}
\usepackage[citecolor=blue]{hyperref}
\usepackage{upgreek}

\usepackage[framemethod=tikz]{mdframed}

\newcommand{\redsout}{\bgroup\markoverwith{\textcolor{red}{\rule[0.5ex]{2pt}{.4pt}}}\ULon}
\newcommand{\e}{\varepsilon}

\newcommand{\curl}{\operatorname{curl}}

\newcommand{\p}{\partial}
\newcommand{\C}{\mathbb{C}}
\newcommand{\R}{\mathbb{R}}
\newcommand{\Z}{\mathbb{Z}}

\newcommand{\supp}{\textup{supp}}

\newcommand{\LC}{\left(}
\newcommand{\RC}{\right)}

\definecolor{skyblue}{rgb}{0.85,0.85,1}

\numberwithin{equation}{section}
\newtheorem{theorem}{Theorem}
\newtheorem{proposition}{Proposition}

\newtheorem{lemma}{Lemma}

\newtheorem{remark}{Remark}

\author[Lai]{Ru-Yu Lai}
\address{School of Mathematics, University of Minnesota, Minneapolis, MN 55455, USA}
\curraddr{}
\email{rylai@umn.edu}

\author[Zhou]{Ting Zhou}
\address{Department of Mathematics, Northeastern University, Boston, MA 02115, USA}
\curraddr{ }
\email{t.zhou@northeastern.edu}
 

\title[Nonlinear Magnetic Schr\"odinger Equation]{Partial Data Inverse Problems for Nonlinear Magnetic Schr\"odinger Equations} 

\date{}

\begin{document}
 
\maketitle
\begin{abstract}
	 We prove that the knowledge of the Dirichlet-to-Neumann map, measured on a part of the boundary of a bounded domain in $\mathbb R^n, n\geq2$, can uniquely determine, in a nonlinear magnetic Schr\"odinger equation, the vector-valued magnetic potential and the scalar electric potential, both being nonlinear in the solution. \\

\end{abstract}

\section{Introduction}

We investigate an inverse boundary value problem for the nonlinear magnetic Schr\"odinger equations. Let $\Omega \subset \mathbb{R}^n,\, n\geq 2$ be an open {connected} bounded domain with smooth boundary $\partial\Omega$, we consider the boundary value problem
\begin{align}\label{intro_eqn}
\LC D + A (x,u) \RC^2 u + q(x,u) =0\qquad \hbox{ in }\Omega, 
\end{align}
with the boundary condition $u=f$ on $\p\Omega$. Here 
the vector-valued function $A = (A_1,\ldots,A_n)$ is the nonlinear magnetic potential, modeling the effect of an external magnetic field, the scalar function $q$ represents the nonlinear electric potential and $D$ denotes $-i\nabla_x$.
The Dirichlet-to-Neumann (DN) map for the equation is defined by 
\begin{equation}\label{intro_DNmap}
\Lambda_{A,q}: W^{2-1/p,p}(\partial\Omega)\rightarrow W^{1-1/p,p}(\partial\Omega), \quad f \mapsto \nu\cdot \left(\nabla u+iA(x,u)u\right)|_{\p\Omega},
\end{equation}
where $\nu$ is the unit outer normal to $\p\Omega$.  

The type of inverse boundary value problem was first formulated by Calder\'on \cite{calderon} for the linear condituctivity equation $\nabla\cdot\gamma(x)\nabla u=0$ when he sought to determine the electrical conductivity $\gamma(x)$ of a medium by making boundary measurements of electric voltage and current. The unique determination was proved in \cite{sylvester1987global} in dimension $n\geq3$ by solving the problem of determining an electric potential $q(x)$ in a Schr\"odinger operator $-\Delta+q$ from the boundary Dirichlet and Neumann data. Since then, the inverse problem has been extensively studied in various generalized cases. The inverse boundary value problem for the linear magnetic Schr\"odinger equation, where $A(x,z)=A(x)$ and $q(x,z)=q(x)$, has been considered in \cite{CR19, Chung2014, ER95, Dos2009, GT2011, Haberman16, IUY2012, KU14, Lai2011, NSU95, PSU10} and the reference therein. 
Specifically, due to a gauge invariance, one can only expect to recover uniquely the magnetic field $\curl A$ and $q$ from the boundary DN-map.

In dealing with the inverse problems for nonlinear PDEs, a standard approach based on the first order linearization of the DN-map was introduced to identify the linear reaction from the medium, then the full nonlinear medium for certain cases.
See for instance \cite{Sun2002, Isakov93, victor01, victorN, isakov1994global, SunUhlmann97} for the demonstration of the approach in solving the inverse problems for certain semilinear, quasilinear elliptic equations and parabolic equations. 
Recently the higher order linearization of the DN-map has been applied in determining the full nonlinearity of the medium for several different equations. 
The method was successfully applied to solve inverse problems for nonlinear hyperbolic equations on the spacetime \cite{KLU18}, where in contrast the underlying problems for linear hyperbolic equations are still open, see also \cite{CLOP,LUW2018} and the references therein. In particular, the second order linearization of the nonlinear boundary map was studied in \cite{CNV19, Kang2002, Sun96, SunUhlmann97} for nonlinear elliptic equations. Moreover, this higher order linearization technique was also applied to study elliptic equations with power-type nonlinearities, see \cite{FO2019, KU201909, KU201905, LaiL2020, LLLS201903, LLLS201905, Li2020_05, Lin202004}. A demonstration of the method can be found in \cite{AZ2018} on nonlinear Maxwell's equations, in \cite{LUY2020} on nonlinear kinetic equations, and in \cite{LLPT2020_1} on semilinear wave equations.

Given a semilinear elliptic PDE whose leading term is the Laplacian operator, we apply the higher order linearization of the DN-map with respect to the small perturbation around the zero solution. The knowledge of the DN-map, measured partially or completely on the boundary, determines an integral of the product of the $m$-th order term of the nonlinear parameter and $m+1$ harmonic functions or their derivatives. A density argument of the products of harmonic functions or their derivatives is crucial in proving the uniqueness of the $m$-th order term. 
For the inverse problem of the linear equation with DN-map measured only on part of the boundary, the density of the product of harmonic functions, which vanish on a closed proper subset of the boundary, was first shown in \cite{FKSjU09}. More specifically, in \cite{FKSjU09}, this density argument relies on a Runge type approximation result and an idea of propagating exponential decay estimates for FBI transforms by the use of maximum principle as in the Kashiwara's watermelon theorem. In \cite{KU201905, LLLS201905}, this density argument was directly used, along with unique continuation and the maximum principle, to show unique determination of a potential function $q(x)$ in a model equation $-\Delta u+q(x)u^2=0$ or in a more general equation of the form $-\Delta u+V(x, u)=0$, assuming partial data. The argument in \cite{FKSjU09} was then generalized in \cite{KU201909} where the authors of \cite{KU201909} proved the density of the products of the gradients of two harmonic functions, which vanish on part of the boundary, and then use it to show the unique determination of a nonlinear potential $q(x)$ in the equation $-\Delta u+q(x)(\nabla u\cdot \nabla u)=0$. \\
 
\subsection{Problem setup and strategy}
In this paper, the main objective is to determine the nonlinear vector potential $A(x,z)$ and the scalar potential $q(x,z)$ in \eqref{intro_eqn} from the boundary DN-map \eqref{intro_DNmap}. 
We briefly state our strategy using the higher order linearization technique as follows. 
 
Suppose that two sets of potentials $(A_1,q_1)$ and $(A_2,q_2)$ satisfy $A_j,\,q_j: \Omega\times\C\rightarrow \C$,
\begin{equation}\label{A_holomorphic}
\hbox{the map }\C \ni z\mapsto A_j(\cdot,z) \hbox{ is holomorphic with values in }W^{1,\infty}(\Omega , \C^n),  
\end{equation}
\begin{equation}\label{q_holomorphic}
\hbox{the map }\C \ni z\mapsto q_j(\cdot,z) \hbox{ is holomorphic with values in }L^{\infty}(\Omega , \C),  
\end{equation}
and 
\begin{align}\label{Aq_conditions}
A_j(x,0)=0,\ \ q_j(x,0)=\p_z q_j(x,0)=0 
\end{align} 
for $j=1,2$. We have that the potentials admit the following expansions:
\begin{equation*}
A_j(x,z)= \sum^\infty_{k=0} \p_z^kA (x,0) {z^k\over k!},\qquad
q_j(x,z)= \sum^\infty_{k=0} \p_z^kq (x,0) {z^k\over k!}. 
\end{equation*}

Fixing a positive integer $m\geq 1$, let $\varepsilon_k$ be small positive numbers and $f_k\in W^{2-1/p,p}(\p\Omega)$ for $k=1,\ldots, m$. 
We denote $\varepsilon := (\varepsilon_1, \ldots, \varepsilon_m)$ and 
$$
u_j := u_j(x;\varepsilon), \qquad j=1, 2
$$
to be the unique small solution of the Dirichlet problem
\begin{align}\label{eqn:MagnSch0}
\left\{\begin{array}{ll}
(D + A_j (x,u_j ))^2 u_j  + q_j (x,u_j ) = 0 & \mbox{ in } \Omega, \\
u_j  = \e_1 f_1 + \ldots + \e_m f_m & \mbox{ on } \partial \Omega.
\end{array}\right.
\end{align}
We establish the well-posedness for this Dirichlet problem with small data in the Appendix and the DN-map is thus well-defined. 
Moreover, by Theorem \ref{Thm:well-posedness} we know that the finite difference $u_j/\varepsilon_k$ is bounded in $W^{2,p}(\Omega)$ (the bound is independent of $\varepsilon$), hence $u_j$ is differentiable in $\varepsilon_k$ and the derivatives satisfy the linearized Laplace equation. (See for example, \cite{Isakov2001, Sun96} for a more detailed exposition.) By expanding $u_j$ in the small perturbation parameter $\varepsilon_k$ and noting that $u_j(x;0)\equiv0$ due to the well-posedness, we have that the first order term
$$
v_{j,k}:=\partial_{\varepsilon_k} u_j |_{\varepsilon=0}
$$ 
is indeed a harmonic function in $\Omega$ satisfying $v_{j,k}|_{\partial\Omega}=f_k$ for $k=1,\ldots,m$, $j=1,2$.
\begin{remark}\label{rk:v_k}
\begin{enumerate} 
	\item We point out that in this setup, we have harmonic functions $v_{1,k}=v_{2,k}$ for $k=1,\ldots,m$ in $\Omega$ since they agree on the whole boundary.
	
	\item In other cases, the domain $\Omega$ might have unknown geometrical features. For example, if there is an unknown inclusion or obstacle embedded in $\Omega$, or if the part of the boundary where we cannot measure the DN-map has an unknown geometry, then we would have $u_j$ to be the solutions to the magnetic Schr\"odinger equation in $\Omega_j$ associated with $A_j$ and $q_j$ for $j=1,2$. This implies that $v_{1, k}$ and $v_{2,k}$ are harmonic functions in potentially different domains $\Omega_j$. These scenarios are discussed in \cite{LLLS201905} for elliptic equations, where one can show under certain assumptions, using unique continuation, that the domains in above examples are indeed identical. 
	
	\item In this paper, we focus on the case where the domain is known to be $\Omega$. For the partial data inverse problems, we assume that the subsets of the boundary: $\Gamma_1$ and $\Gamma_2$, be where Dirichlet data and Neumann data are measured respectively. By the definition of the partial DN-map in \eqref{eqn:DNmap_p}, we have the harmonic function $v_{j,k}=\partial_{\varepsilon_k} u_j |_{\varepsilon=0}$, $j=1,2$ satisfying the boundary condition
	$$v_{1,k}|_{\p\Omega}=v_{2,k}|_{\p\Omega}=f_k\qquad \hbox{with }\supp(f_k)\subset \Gamma_1.$$ Therefore, in the partial data setting, we still have $v_{1,k}=v_{2,k}$ in $\Omega$, hence we simply denote $$v_k:=v_{j,k}\qquad \hbox{for }k=1,\ldots, m$$ from this point on. 
\end{enumerate} 
\end{remark}

To reconstruct $A(x,z)$ and $q(x,z)$, it is sufficient to consider the unique determination of $\p_z A(x,0), \p^2_z A(x,0), \ldots$ and $\p^2_z q(x,0), \p^3_zq(x,0), \ldots$ in $\Omega$ due to \eqref{Aq_conditions}. The proof is based on induction steps and is sketched as follows. We start with the second order linearization. Let $m=2$ and denote
$$
w_j :=\partial_{\varepsilon_1}\p_{\varepsilon_2} u_j|_{\varepsilon=0}.
$$ 
Then $w_j$ is the solution to the problem
\begin{align}\label{intro:w}
-\Delta w_j  +Q_j^{(2)}(v_1,v_2)=0\quad \textrm{ in }\Omega,\qquad w_j|_{\p\Omega}=0,
\end{align}
where $v_k=v_{j,k}$ for $k=1,2$, as discussed in Remark \ref{rk:v_k}, and 
\begin{equation}\label{Q2}
Q_j^{(2)}(v_{1}, v_{2}):=3\partial_zA_j(x,0)\cdot \left(v_1 D v_2+v_2 Dv_1\right) + 2D_x\cdot \partial_zA_j(x, 0)v_1v_2+ \p_z^2q_j(x,0)v_1 v_2,
\end{equation}
defined in \eqref{eqn:Q}.
Assume that the DN-maps associated to $(A_1,q_1)$ and $(A_2,q_2)$ are identical. We will obtain the integral identity 
\begin{equation*}\int_\Omega \left(Q_1^{(2)}(v_1, v_2)-Q^{(2)}_2(v_1, v_2)\right)v_3~dx=0,\end{equation*}
where $v_3$ is a third harmonic function in $\Omega$ with certain boundary condition. 
One can see that the integral involves several complicated terms of products of harmonic functions and their gradients, as well as mixtures of the vector and scalar potentials, unlike the cases studied in \cite{KU201909, LLLS201905}. 

In the spirit of \cite{FKSjU09, KU201909}, one can potentially use the corrected harmonic exponentials 
\[v(x,\zeta)=e^{-ix\cdot\zeta \over h}+w(x,\zeta), \qquad \zeta\in\mathbb C^n,\;\zeta\cdot\zeta=0\]
that vanishes on a closed proper subset of the boundary, the idea of propagating exponential decay estimates for FBI transforms and a proper version of Runge-type approximation, to prove an improved density result. However, the exponential decay propagation is difficult to derive for the associated FBI type transform of the vector-valued potential (multiplied by the complex phase). Another major difficulty comes from the entanglement of $A$ and $q$ in the mixture of terms. 

Instead, we combine the previously established density result in \cite{FKSjU09} and the corrected harmonic exponentials together to obtain the local uniqueness of the potentials. Then we conduct the local-to-global step, as in the previous work, using the $H^1$ Runge-type approximation. Our key step here lies on a transport equation for the harmonic functions, which helps decouple the potentials. 

The argument can be easily generalized to the case $m>2$ by induction.

\subsection{Main result}

Let us present our main result where we show that partial data on the boundary is sufficient to uniquely determine the nonlinear potentials in the magnetic Schr\"odinger equation. Meanwhile, for the completeness of the paper, we also provide a separate proof for the situation with full data in the Appendix.

Let $\Omega \subset \mathbb{R}^n$ be an open {connected} bounded domain with smooth boundary $\partial\Omega$.  Let $u$ be the solution to the boundary value problem for the magnetic Schr\"odinger equation with nonlinearity:
\begin{align}\label{eqn:MagnSch}
\left\{\begin{array}{ll}
\LC D + A (x,u) \RC^2 u + q(x,u) =0 & \hbox{in }\Omega,\\
u=f & \hbox{on }\p\Omega,\\
\end{array}\right.
\end{align}
where {$A(x,z) \in W^{1,\infty}(\Omega\times\mathbb C,\mathbb C^n)$ and $q(x,z)\in L^\infty(\Omega\times\mathbb C,\mathbb C)$ are both $C^\infty$ in $z$}, and $D :=-i\nabla$.
Assume that $A$ and $q$ satisfy \eqref{A_holomorphic}-\eqref{Aq_conditions}. 
We will show that the Dirichlet problem \eqref{eqn:MagnSch} has a unique solution $u\in W^{2,p}(\Omega)$ for sufficiently small boundary condition $f\in W^{2-1/p,p}(\p\Omega)$ where $p>n$. It is clear that the equation \eqref{eqn:MagnSch} with $f=0$ admits the zero solution $u=0$. Then the full boundary DN-map for such small functions is defined by 
\begin{equation}\label{eqn:DNmap}
\Lambda_{A,q}: W^{2-1/p,p}(\partial\Omega)\rightarrow W^{1-1/p,p}(\partial\Omega), \quad f \mapsto \nu\cdot \left(\nabla u+iA(x,u)u\right)|_{\p\Omega},
\end{equation}
where $\nu$ is the unit outer normal to $\p\Omega$. 
We also define the partial boundary DN-map as follows. Let $\Gamma_1, \Gamma_2 \subset \p\Omega$ be two arbitrary, nonempty open subsets. Then the partial boundary DN-map is defined by
\begin{equation}\label{eqn:DNmap_p}\Lambda_{A,q}^{\Gamma_1,\Gamma_2}(f)=\Lambda_{A,q}(f)|_{\Gamma_2}\qquad \textrm{ for all } f\in W^{2-1/p,p}(\p\Omega) \textrm{ with } \supp (f)\subset \Gamma_1.\end{equation}

\begin{theorem}\label{thm:global uniqueness}
Let {$\Omega\subset \R^n,\ n\geq2$} be an open {connected} bounded domain with $C^\infty$ boundary $\p\Omega$ and
let $\Gamma_1,\Gamma_2\subset \p\Omega$ be arbitrary nonempty open subsets of $\p\Omega$. {Suppose that two sets of coefficients $(A_1, q_1)$ and $(A_2, q_2)$ satisfy \eqref{A_holomorphic}-\eqref{Aq_conditions}} and
\[\nu\cdot \p_z^kA_1(x,0)=\nu\cdot \p_z^kA_2(x,0)\qquad \textrm{ on }\Gamma_1\cap\Gamma_2 \qquad\hbox{ for }{k\geq 1}.\]
Let $\Lambda^{\Gamma_1,\Gamma_2}_{A_j,q_j}$ be the above partial boundary DN-map associated to $(A_j, q_j)$ for $j=1,2$.
Suppose that $\Lambda^{\Gamma_1,\Gamma_2}_{A_1, q_1}(f)=\Lambda^{\Gamma_1,\Gamma_2}_{A_2,q_2}(f)$ for any $f\in W^{2-1/p,p}(\p\Omega)$, $n<p<\infty$ with $\supp (f)\subset\Gamma_1$ and $\|f\|_{W^{2-1/p,p}(\p\Omega)}<\delta$, where $\delta>0$ is a sufficiently small constant. 
Then  
$$
     A_1 =A_2\ \ \hbox{and }\ q_1=q_2\qquad \hbox{ in }\Omega.
$$
\end{theorem}

\begin{remark}
When $\Gamma_1=\Gamma_2=\partial\Omega$, this is the uniqueness result for the inverse problem with full boundary data. In particular, it can be showed by a separate and direct method as seen in the Appendix.   
\end{remark}

We comment here due to the assumption \eqref{Aq_conditions}, the first order linearization of the DN-map provides boundary measurements of the harmonic functions in $\Omega$. As commented in Remark \ref{rk:v_k}, we could adopt the argument in \cite{LLLS201905} to show the unique determination of obstacles embedded in $\Omega$ or the unknown geometry of the inaccessible part of the boundary.

Another important observation is that our result shows that there is no gauge invariance for this problem.

The paper is organized as follows. The higher order linearization technique is detailed in Section~\ref{sec:linear} and the crucial integral identity is also derived there. Then the proof of Theorem \ref{thm:global uniqueness} is given in Section \ref{sec:partial data}. 
The well-posedness for the boundary value problem of the nonlinear magnetic Schr\"odinger equation is established in Appendix~\ref{sec:pre}. 
Finally, an alternative proof of the uniqueness of potentials with full boundary measurements is provided in Appendix \ref{append:alter_pf}.

\section*{Acknowledgements}
R.-Y. Lai is partially supported by the NSF grant DMS-1714490. T. Zhou was working on the project during visiting the MSRI in participation of the 2019 semester program on Microlocal Analysis and would like to thank Prof. Gunther Uhlmann and Prof. Katya Krupchyk for helpful discussions. 
Both authors thank Prof. Francis Chung for his tremendous help with the draft and constructive advice.


\section{The higher order linearization}\label{sec:linear}
 
In this section, we use the higher order linearization approach to derive a key integral identity encoding the information of the discrepancy of the potentials $A$ and $q$, as stated in Proposition~\ref{HigherIntegralIdentity}. We start by considering the $m=2$ case and then extend it to the higher order terms by induction steps.

For $m\geq 2$, let $\varepsilon := (\varepsilon_1, \ldots, \varepsilon_m)$ with $\varepsilon_k>0$ and let $f_k\in W^{2-1/p,p}(\p\Omega)$ with $\supp(f_k)\subset\Gamma_1$, $k=1,\ldots, m$. 
Under the assumptions of Theorem \ref{thm:global uniqueness}, the boundary value problem
\begin{align}\label{eqn:MagnSch1}
\left\{\begin{array}{ll}
(D + A_j(x,u ))^2 u_j  + q_j(x,u_j ) = 0 & \mbox{ in } \Omega, \\
u_j  = \e_1 f_1 + \ldots + \e_m f_m & \mbox{ on } \partial \Omega,
\end{array}\right.
\end{align}
admits a unique solution $u_j=u_j(x;\varepsilon)$ for $|\varepsilon|$ small enough.

\subsection{For $m=2$ case}
We recall the condition \eqref{Aq_conditions}.
Given the boundary condition $f=\varepsilon_1 f_1+\varepsilon_2 f_2$ with $\supp(f_k)\subset\Gamma_1$ for small enough $|\varepsilon|$, following the steps described in the introduction, 
the first order linearization of \eqref{eqn:MagnSch1} around the zero solution $u_j(x;0)=0$ gives that $v_{j,k}:=\partial_{\varepsilon_k}u_j(x;\varepsilon)|_{\varepsilon=0}$, $k=1,2$, is harmonic function satisfying
\begin{equation}\label{eqn:1nd_Linearization}
-\Delta  v_{j,k}=0\quad\textrm{ in }\Omega,\qquad v_{j,k}|_{\partial\Omega}=f_k.
\end{equation} 
This indeed implies that 
$$
   v_k:=v_{1,k}=v_{2,k}\quad\textrm{ in }\Omega.
$$
Next we perform the second order linearization, then it gives that the function
$$w_j:=\partial_{\varepsilon_1}\partial_{\varepsilon_2}u_j(x;\varepsilon)|_{\varepsilon=0} $$ 
is the solution to
\begin{equation}\label{eqn:2nd_Linearization}
-\Delta w_j+Q^{(2)}(v_1, v_2) =0\quad\textrm{ in }\Omega,\qquad w_j|_{\partial\Omega}=0,
\end{equation}
where 
\begin{equation}\label{eqn:Q}Q^{(2)}(v_1, v_2):=3\partial_zA_j(x,0)\cdot \left(v_1 D v_2+v_2 Dv_1\right) + 2D_x\cdot \partial_zA_j(x, 0)v_1v_2+\p_z^2q_j(x,0)v_1 v_2,
\end{equation}
with the partial $D_x$ meaning the derivative with respect to the first variable of $A_j(x, u)$. 
Then the $O(\varepsilon_1\varepsilon_2)$ term in the expansion of the DN-map is 
\begin{equation}\label{eqn:DN_w}\partial_{\varepsilon_1}\partial_{\varepsilon_2}|_{\varepsilon=0}[\Lambda^{\Gamma_1, \Gamma_2}_{A_j,q_j}(\varepsilon_1f_1+\varepsilon_2f_2)]=\left(\p_\nu w_j+2i\nu\cdot\partial_zA_j(x,0)f_1f_2\right)|_{\Gamma_2}.\end{equation}  

We then have the integral identity in the $m=2$ case.

\begin{proposition}\label{lem:int_identity}
Let {$\Omega\subset \R^n,\ n\geq2$} be an open {connected} bounded domain with $C^\infty$ boundary $\p\Omega$ and
let $\Gamma_1,\Gamma_2\subset \p\Omega$ be arbitrary nonempty open subsets of $\p\Omega$.
Given two sets of potentials $(A_1, q_1)$ and $(A_2, q_2)$ that satisfy the conditions \eqref{A_holomorphic}-\eqref{Aq_conditions} and 
\[{\nu\cdot\p_zA_1(x,0)=\nu\cdot\p_zA_2(x,0)\qquad \textrm{ on }\Gamma_1\cap\Gamma_2},\]
we then have that if $\Lambda^{\Gamma_1,\Gamma_2}_{A_1,q_1}=\Lambda^{\Gamma_1, \Gamma_2}_{A_2,q_2}$ (for small boundary data), then for any harmonic functions $v_1, v_2, v_3$ with 
$$
\supp (v_1|_{\p\Omega}),\, \supp (v_2|_{\p\Omega})\subset\Gamma_1\ \hbox{ and }\ \supp (v_3|_{\p\Omega})\subset\Gamma_2,$$
we have
\begin{equation}\label{InteriorIntegral}\int_\Omega \left(Q_1^{(2)}(v_1, v_2)-Q^{(2)}_2(v_1, v_2)\right)v_3~dx=0,\end{equation}
where $Q_j^{(2)}(v_1, v_2)$ is given by \eqref{eqn:Q} with $A, q$ replaced by $A_j, q_j$ for $j=1,2$.
\end{proposition}

\begin{proof}
Let $v_1$ and $v_2$ be harmonic functions with boundary conditions $f_k:=v_k|_{\p\Omega}$ and $\supp(f_k)\subset \Gamma_1$ for $k=1,2$. 
From the fact that $\Lambda^{\Gamma_1,\Gamma_2}_{A_1, q_1}(\varepsilon_1f_1+\varepsilon_2 f_2)=\Lambda^{\Gamma_1,\Gamma_2}_{A_2,q_2}(\varepsilon_1f_1+\varepsilon_2f_2)$ for small $\varepsilon=(\varepsilon_1, \varepsilon_2)$, we have that 
$$
\p_\nu w_1+2i\nu\cdot\partial_zA_1(x,0)f_1f_2=\p_\nu w_2+2i\nu\cdot\partial_zA_2(x,0)f_1f_2\qquad \textrm{ on }\Gamma_2,
$$
where $w_1, w_2$ are solutions to 
\begin{equation}\label{eqn:linear w}-\Delta w_j+Q^{(2)}_j(v_1, v_2)=0 \quad \textrm{ in }\Omega,\qquad w_j|_{\p\Omega}=0.\end{equation}
Since  $\supp(f_1),\supp(f_2)\subset \Gamma_1$ and $\nu\cdot\p_zA_1(x,0)=\nu\cdot\p_zA_2(x,0)$ on $\Gamma_1\cap\Gamma_2$, one has $$\nu\cdot\p_zA_1(x,0)f_1f_2=\nu\cdot\p_zA_2(x,0)f_1f_2 \qquad \hbox{ on $\Gamma_2$}, $$  which leads to 
$$
    \p_\nu w_1 |_{\Gamma_2} = \p_\nu w_2|_{\Gamma_2}.
$$
Multiplying \eqref{eqn:linear w} by any harmonic function $v_3$ in $\Omega$ with $\supp(v_3|_{\p\Omega})\subset \Gamma_2$ and applying Green's formula, we then derive that
\begin{align*} 
\int_\Omega \left(Q^{(2)}_1(v_1, v_2)-Q^{(2)}_2(v_1, v_2)\right)v_3 ~dx =\int_{\p\Omega\setminus \Gamma_2} (\p_\nu w_1-\p_\nu w_2)v_3~dS=0.
\end{align*}
Thus, the proof is complete.
\end{proof}

\subsection{Induction steps in $m\geq2$.}	
Let $m\geq 2$ and suppose that
$$
\p_z^kA_1(x,0)=\p_z^kA_2(x,0)\ \ \hbox{for } k=1,\ldots, m-2,
$$
$$
\p_z^kq_1(x,0)=\p_z^k q_2(x,0)\ \ \hbox{for } k=2,\ldots, m-1.
$$
Combining the base case \eqref{Aq_conditions}, we have that 
\begin{equation}\label{AAssumption}
\partial_z^k A_1(x,0) = \partial_z^k A_2(x,0) \mbox{ for } k = 0, \ldots, m-2,
\end{equation}
and
\begin{equation}\label{QAssumption}
\partial_z^k q_1(x,0) = \partial_z^k q_2(x,0) \mbox{ for } k = 0, \ldots, m-1.
\end{equation}
Let $\varepsilon=(\varepsilon_1,\ldots, \varepsilon_m)$ with small enough $\varepsilon_k>0$ and let $f_k\in W^{2-1/p,p}(\p\Omega)$ with $\supp(f_k)\subset \Gamma_1$ for $k=1,\ldots,m$. 
Again, from Theorem~\ref{Thm:well-posedness}, there exists a unique small solution $u_j = u_j(x;\e)$ to the problem	
\begin{align}\label{EpsilonmEquation}
\left\{\begin{array}{ll}
    (D + A_j(x,u_j))^2 u_j + q_j(x,u_j) = 0 & \mbox{ in } \Omega, \\
    u_j = \e_1 f_1 + \ldots + \e_m f_m & \mbox{ on } \partial \Omega,
\end{array}\right.
\end{align}
for $j = 1,2$.

Generally, for any positive integer $m\geq 2$, we define the function $Q_j^m(v_1, \ldots, v_m)$ by
\begin{align}\label{def Q}
\begin{split}
Q_j^{(m)}(v_1, \ldots, v_m)  
&:=(m+1)\p_z^{m-1}A_j(x,0)\cdot D(v_1\ldots v_m)+m(D_x\cdot \p_z^{m-1}A_j(x,0))v_1\ldots v_m\\
&\quad +\p_z^mq_j(x,0)v_1\ldots v_m.
\end{split}
\end{align}

The general integral identity is summarized in the following proposition:
\begin{proposition}\label{HigherIntegralIdentity}
Let $(A_1, q_1)$ and $(A_2, q_2)$ satisfy the conditions \eqref{A_holomorphic}-\eqref{Aq_conditions}.
Moreover, suppose for $m\geq 2$, \eqref{AAssumption} and \eqref{QAssumption} are satisfied and
\begin{equation}
    \label{A_bdry_cond}\nu\cdot \p_z^{m-1} A_1(x,0)=\nu\cdot \p_z^{m-1} A_2(x,0) \qquad \textrm{ on }\Gamma_1\cap\Gamma_2
\end{equation}
holds. 
If $\Lambda_{A_1,q_1}^{\Gamma_1,\Gamma_2}=\Lambda_{A_2,q_2}^{\Gamma_1,\Gamma_2}$ for small data, then for any harmonic functions $v_1, \ldots, v_{m+1}$ satisfying $$\supp (v_1|_{\p\Omega}), \ldots,\ \supp (v_m|_{\p\Omega})\subset \Gamma_1\ \hbox{ and }\ \supp (v_{m+1}|_{\p\Omega})\subset \Gamma_2,$$ we have
\begin{align}\label{keyId}
\int_{\Omega} \left(Q_1^{(m)}(v_1, \ldots, v_m) - Q_2^{(m)}(v_1, \ldots, v_m)\right) v_{m+1} \, dx = 0.
\end{align}
\end{proposition}

Before proving Proposition~\ref{HigherIntegralIdentity}, we first need to look more closely the derivative of $u_j$ with respect to $\varepsilon$, which is stated in Lemma~\ref{EjuLemma}.

We start with defining the notation $E_k$ for $1\leq k\leq m$ to be a product of $k$ distinct operators of the form $\partial_{\e_i}$ and setting $\e=0$. For example, for distinct numbers $1\leq \ell_1,\ell_2,\ldots,\ell_k\leq m$, we have that $\p_{\e_{\ell_1}}\p_{\varepsilon_{\ell_2}}\ldots\p_{\varepsilon_{\ell_k}}u|_{\e =0}  $ is a representative of $E_ku$.

\begin{lemma}\label{EjuLemma} 
Suppose that \eqref{AAssumption} and \eqref{QAssumption} hold.  
Let $u_j$ be the solution to \eqref{EpsilonmEquation}. For any $1 \leq k < m$,
then we have 
\[
E_k u_1 = E_k u_2.
\]
\end{lemma}
 
\begin{proof}
To demonstrate this for $k = 1$, we apply $\partial_{\e_i}$, $1\leq i\leq m$, to the equation \eqref{EpsilonmEquation} for $u_j$ and set $\e = 0$.  As before we find $\partial_{\e_i}u_j|_{\e = 0}$ satisfies the linear equation
\[
\Delta \partial_{\e_i} u_j|_{\e =0} = 0, 
\]
with boundary condition $(\p_{\e_i}u_j|_{\e=0})|_{\p\Omega} = f_i$.  Since this holds for $j = 1,2$, we conclude that $E_1 u_1 = E_1 u_2$.  

Now we proceed by the induction argument. Suppose that
\begin{equation}\label{EiuAssumption}
E_i u_1 = E_i u_2
\end{equation}
holds for any $E_i$ with $1 \leq i < k$, and we want to show $E_k u_1=E_ku_2$. Without loss of generality, we consider the operator
\[
E_k = \partial_{\e_1} \ldots \partial_{\e_k}|_{\e=0}. 
\]
By applying $E_k$ to \eqref{EpsilonmEquation}, we have 
\[0=-\Delta E_ku_j+\Psi_k(u_j, A_j, q_j),\]
where the term $\Psi_k$ is defined by
\begin{align*}
 & \Psi_k(u_j, A_j, q_j)\\
 &:=\Psi_k(E_1u_j,\ldots, E_{k-1}u_j, \p_z^1A_j(x,0),\ldots, \p_z^{k-1}A_j(x,0), \p_z^2 q_j(x,0), \ldots, \p_z^kq_j(x,0))
\end{align*}  
and contains derivatives of order at most $k-1$ in $u_j$, derivatives of order at most $k-1$ in $A_j$, and derivatives of order at most $k$ in $q_j$ with respect to the variable $z$.  
Since $k < m$, we have $k \leq m-1$ and $k - 1 \leq m-2$. Therefore combining \eqref{AAssumption}, \eqref{QAssumption}, and \eqref{EiuAssumption}, we have 
\[
\Psi_k(u_1, A_1,q_1) = \Psi_k(u_2, A_2,q_2).
\]
Therefore the conclusion is that 
\[
-\Delta (E_k u_1) = -\Delta (E_k u_2).
\]
Moreover, $E_k u_1$ and $E_k u_2$ share the same boundary condition, then one can conclude
\[
E_k u_1 = E_k u_2\]
as desired. 

\end{proof}

\begin{proof}[Proof of Proposition \ref{HigherIntegralIdentity}]

For the given harmonic functions $v_1,\ldots, v_m$, whose boundary traces are supported on $\Gamma_1$, we consider the solution $u_j$ to \eqref{EpsilonmEquation} with $f_k=v_k|_{\p\Omega}$ ($k=1,\ldots,m$). It is not hard to see by Lemma \ref{EjuLemma} that $\p_{\e_k}u_1|_{\e=0}=\p_{\e_k} u_2|_{\e=0}=v_k$.

Applying the operator $\partial_{\e_1} \ldots \partial_{\e_m}|_{\e=0}$ to \eqref{EpsilonmEquation}, we get
\begin{align*}
0 &=  -\Delta (\partial_{\e_1} \ldots \partial_{\e_m} u_j|_{\e=0}) \\
 &\quad +(m+1)\p_z^{m-1}A_j(x,0)\cdot D(v_1\ldots v_m)+m(D_x\cdot \p_z^{m-1}A_j(x,0))v_1\ldots v_m\\
 &\quad +\p_z^mq_j(x,0)v_1\ldots v_m+R_m(u_j, A_j, q_j).
\end{align*}
Here the remaining term $R_m$ contains derivatives of order at most $m-1$ in $u_j$, at most $m-2$ in $A_j(x,0)$, and at most $m-1$ in $q_j(x,0)$ with respect to $z$ variable.  For $j=1,2$, if we write $$\phi_j:=\partial_{\e_1} \ldots \partial_{\e_m} u_j|_{\e=0}$$ and use
the notation introduced in \eqref{def Q}, then
we can write this as 
\[
\Delta \phi_j = Q^{(m)}_j(v_1, \ldots, v_m)  +R_m(u_j, A_j, q_j) \quad \textrm{ in }\Omega,\qquad \phi_j|_{\p\Omega}=0.
\]
From Lemma \ref{EjuLemma} ($E_\ell u_1=E_\ell u_2$, $1\leq \ell < m$) and the assumptions on $A_j$ and $q_j$, we see that 
\[
R_m(u_1, A_1, q_1) = R_m(u_2, A_2, q_2).
\]
Therefore if we subtract the equation for $j = 2$ from the equation for $j = 1$, we get
\[
\Delta (\phi_1 - \phi_2) = Q^{(m)}_1(v_1, \ldots, v_m) - Q^{(m)}_2(v_1, \ldots, v_m) .
\]
From the equality of the DN-maps $\Lambda_{A_1,q_1}^{\Gamma_1,\Gamma_2}=\Lambda_{A_2,q_2}^{\Gamma_1,\Gamma_2}$, Lemma~\ref{EjuLemma} ($E_ku_1=E_ku_2$, $1\leq k< m$), and the condition \eqref{A_bdry_cond}, we can easily derive 
\begin{equation}\label{w_bdry_gamma_2}\p_\nu \phi_1|_{\Gamma_2}=\p_\nu \phi_2|_{\Gamma_2}.\end{equation}
Now let $v_{m+1}$ be harmonic and $\supp (v_{m+1}|_{\p\Omega})\subset\Gamma_2$, and consider the integral
\[
\int_{\Omega} \Delta (\phi_1 - \phi_2)v_{m+1}~dx. 
\]
Similar to the case $m=2$ discussed in the proof of Proposition~\ref{lem:int_identity}, by performing the integration by parts, we get 
\[
\int_{\Omega} \Delta (\phi_1 - \phi_2)v_{m+1} \, dx= \int_{\Omega} (\phi_1 - \phi_2)\Delta v_{m+1}~dx=0,
\]
with no boundary terms, thanks to the equality \eqref{w_bdry_gamma_2}, $\phi_1|_{\p\Omega}=\phi_2|_{\p\Omega}=0$ and $\supp (v_{m+1}|_{\p\Omega})\subset\Gamma_2$. This gives us 
\[
\int_{\Omega} \left(Q^{(m)}_1(v_1, \ldots, v_m) - Q^{(m)}_2(v_1, \ldots, v_m) \right) v_{m+1}~dx = 0.
\]
This finishes the proof.
\end{proof}

From Proposition~\ref{lem:int_identity} and Proposition~\ref{HigherIntegralIdentity}, we have proved that the integral identity \eqref{keyId} holds for $m\geq 2$. In the next section, we will focus on extracting the information of potentials $A$ and $q$ from this identity.

\section{Proof of Theorem~\ref{thm:global uniqueness}}\label{sec:partial data}

\subsection{A key lemma}
We will see below that by using the integral identity \eqref{keyId} in Proposition~\ref{HigherIntegralIdentity} and the density result in Theorem~1.1 in \cite{FKSjU09}, we can derive a much simpler identity that will be the key component to show the desired uniqueness result.

To this end, we first simplify the notations by denoting the discrepancy in $\p_z^{m-1} A_1$ and $\p_z^{m-1} A_2$ as
$$\widetilde{A}_{m-1} (x):= \p_z^{m-1} A_2(x,0) - \p_z^{m-1}A_1 (x,0),
$$
and also denoting the discrepancy in $\p_z^{m} q_1$ and $\p_z^{m} q_2$ as
$$
\widetilde{q}_m(x) := \p_z^m q_2(x,0) - \p_z^m q_1 (x,0).
$$

By applying the integration by parts and the boundary condition $\nu\cdot \p_z^{m-1}A_1(x,0)=\nu\cdot \p_z^{m-1} A_2(x,0)$, the identity in Proposition~ \ref{HigherIntegralIdentity} now becomes
\[\begin{split} 0 
&=\int_\Omega \LC Q_2^{(m)}(v_1,\ldots, v_m)-Q_1^{(m)}(v_1,\ldots, v_m)\RC 	v_{m+1}~dx\\
&=\int_\Omega \Big(-(D\cdot  \widetilde A_{m-1})  v_{m+1}-(m+1) ( \widetilde A_{m-1}\cdot Dv_{m+1})  +\widetilde q_m v_{m+1}\Big)	 v_1\ldots v_{m }~dx
\end{split}
\] 
for harmonic functions $v_1, \ldots, v_{m+1}$ such that $$\supp (v_1|_{\p\Omega}), \ldots, \supp (v_m|_{\p\Omega})\subset \Gamma_1\quad \hbox{ and }\quad \supp (v_{m+1}|_{\p\Omega})\subset \Gamma_2.$$ 

We then have the following result by using the density of the product of harmonic functions in $L^1$ space.
\begin{lemma}\label{lemma transport}
Suppose that the conditions in Proposition~\ref{HigherIntegralIdentity} hold and harmonic functions $v_3,\ldots,v_{m}$ have nontrivial boundary data. Then the following identity holds:
	 \begin{equation}\label{eqn:v_mplus1}
	-(D\cdot \widetilde A_{m-1})  v_{m+1}-(m+1) ( \widetilde A_{m-1}\cdot Dv_{m+1})  +\widetilde q_mv_{m+1}=0
	 \end{equation}
	 almost everywhere \textup{(}a.e.\textup{)} in $\Omega$ for $m\geq 2$.  
\end{lemma}
\begin{proof}
When $m=2$, we apply Theorem~1.1 in \cite{FKSjU09}, stating that the set of products $v_1v_2$ of harmonic functions in $C^\infty(\overline\Omega)$ that vanish on a closed proper subset $\widetilde\Gamma$ of $\p\Omega$ is dense in $L^1(\Omega)$. We immediately obtain
\[-(D\cdot \widetilde A_1)v_3-3( \widetilde A_1\cdot Dv_3)+\widetilde q_2 v_3=0\]
a.e. in $\Omega$.

{For $m>2$, we also apply Theorem 1.1 in \cite{FKSjU09} to obtain 
\[\left(-(D\cdot  \widetilde A_{m-1} )  v_{m+1}-(m+1) ( \widetilde A_{m-1} \cdot Dv_{m+1})  + \widetilde q_mv_{m+1}\right)v_3\ldots v_m=0\]
a.e. in $\Omega$. 
By the fact that $\bigcup_{j=3}^m v_j^{-1}(0)$ has measure zero, we then have \eqref{eqn:v_mplus1} for any $m>2$.
}
	\end{proof}

For simplicity of notations, for $m\geq 2$, we recast \eqref{eqn:v_mplus1} as the following transport equation
\begin{align}\label{Fg_ID}
F(x)\cdot Dv_{m+1}+g(x)v_{m+1}=0,
\end{align}
where 
\begin{align}\label{Fg}
F(x):= -(m+1)\widetilde A_{m-1} (x),\ \ g(x):=-(D\cdot  \widetilde A_{m-1} )(x) +  \widetilde q_m(x).
\end{align}

\subsection{Uniqueness result} 
The proof of Theorem~\ref{thm:global uniqueness} heavily stands on the following result, that is, Proposition~\ref{prop:Global result} below: if the equation
\[F(x)\cdot Dv_{m+1}+g(x)v_{m+1}=0\qquad \hbox{ in }\Omega\]
holds for all harmonic functions $v_{m+1}\in C^\infty(\overline{\Omega})$ with $\supp (v_{m+1}|_{\p\Omega})\subset \Gamma_2\subset\p\Omega$, then $F=0$ and $g=0$ in $\Omega$.
It is clear that the identity $F=0=g$ implies $\widetilde A_{m-1}=0= \widetilde q_m$. 
Thus, the uniqueness of the potentials follows immediately.

The key strategy is to first	 show that $F$ and $g$ vanish locally, that is, $F(x)$ and $g(x)$ vanish a.e. in a neighborhood of a point $x_0\in \Gamma_2$. Next we extend this local result to the global one. The detailed argument is stated in the proof of Proposition~\ref{prop:Global result}. 

To begin, we first construct the harmonic function $v_{m+1}$ as in \cite{FKSjU09}. Without loss of generality, we let $x_0=0$, the tangent plane to $\p\Omega$ at $x_0$ be given by $x_1=0$ and
\[\Omega\subset\{x\in \R^n:~|x+{\bf e}_1|<1\},\quad \widetilde\Gamma_2:=\p\Omega\backslash\Gamma_2=\{x\in\p\Omega: x_1\leq -2c\}\] 
for some constant $c>0$. Here 
$$
{\bf e}_j =(0,\ldots,0,1,0,\ldots,0)
$$ with the $j^{th}$ component equals to $1$.

\begin{remark}
We comment here that one can apply a transformation to achieve above conditions for the domain  (or by the conformal mapping in \cite{FKSjU09} for non-convex domains). More specifically, given the transformation $T: \widetilde\Omega \rightarrow \Omega$, $x=T(y)$, the transport equation 
\[F(x)\cdot Dv+g(x) v=0\qquad\textrm{ in }\;\Omega\]
becomes
\[\widetilde{F}(y)\cdot \LC{\p y\over \p x}\RC^{T}D_y\tilde v+\tilde{g}	(y)\tilde v=0\qquad\textrm{ in }\;\tilde\Omega,\]
where  
\[\tilde v(y)= v\circ T(y),\, \widetilde F(y)= F\circ T(y),\,\tilde g(y)= g\circ T(y).\]
It is not hard to see that with the transformation satisfying $\det\LC{\p y\over \p x}\RC\neq 0$ a.e., we have that $\LC{\p y\over \p x}\RC^{T}\widetilde F(y)=0$ and $\tilde g(y)=0$ in $\widetilde\Omega$ implies that $F(x)=0$ and $g(x)=0$ in $\Omega$.
\end{remark}

For $\zeta\in\C^n$ such that $\zeta\cdot\zeta=0$ and a cut-off function $\chi\in C_0^\infty(\R^n)$ such that $\chi=1$ on $\widetilde\Gamma_2$ and $\supp(\chi)\subset \{x\in\R^n:~x_1\leq -c\}$, we consider the harmonic function 
\begin{align}\label{fun vm+1}
    v_{m+1}(x,\zeta)=e^{-ix\cdot\zeta/h}+\widetilde w(x,\zeta),\qquad h>0,
\end{align}
with $v_{m+1}|_{\widetilde{\Gamma}_2}=0$, where $\widetilde w$ is the solution to the Dirichlet problem
\[\left\{\begin{split}&\Delta \widetilde w=0\qquad\textrm{in }\Omega,\\& \widetilde w|_{\p\Omega}=-\left(e^{-ix\cdot\zeta/h}\chi\right)|_{\p\Omega}.\end{split}\right.\]
Then it is clear to see
\[\|\widetilde w\|_{H^1(\Omega)}\leq C\|e^{-ix\cdot\zeta/h}\chi\|_{H^{1/2}(\p\Omega)}\leq C(1+h^{-1}|\zeta|)^{1/2}e^{\frac1h H_K(\operatorname{Im}\zeta)},\]
where $H_K$ is the supporting function of the compact subset $K=\supp(\chi)\cap\p\Omega$ of the boundary and is defined by
\[H_K(\vec d)=\sup_{x\in K}x\cdot\vec d,\qquad \textrm{ for }\vec d\in \R^n.\]
In particular, from the property of $\chi$, one can further derive that when $\operatorname{Im}\zeta_1\geq0$,
\[\|\widetilde w\|_{H^1(\Omega)}\leq C(1+h^{-1}|\zeta|)^{1/2}e^{-\frac c h \operatorname{Im}\zeta_1}e^{\frac1h|\operatorname{Im}\zeta'|},\]
where $\zeta=(\zeta_1, \zeta')$ with $\zeta'$ being the $(n-1)$ dimensional coordinate vector. 

Similarly, we can also derive that for any $\alpha\in \Z^n_+$, the remainder function $\widetilde{w}$ satisfies
\[\begin{split}\|\p^\alpha\widetilde w\|_{H^1(\Omega)}&\leq C\|\p^\alpha (e^{-ix\cdot\zeta/h}\chi)|_{\p\Omega}\|_{H^{1/2}(\p\Omega)}\\
&\leq C(1+h^{-1}|\zeta|)^{(1+|\alpha|)/2}e^{-\frac c h \operatorname{Im}\zeta_1}e^{\frac1h|\operatorname{Im}\zeta'|},
\end{split}\]
which gives the upper bound of $\widetilde{w}\in H^{|\alpha|+1}(\Omega)$. By the Sobolev embedding theorem \cite{Evan}, when $|\alpha|-[\frac n 2]-1\geq 0$,  
one has
\begin{align}\label{diff w}
 \|\widetilde w\|_{C^1(\overline\Omega)}\leq C(1+h^{-1}|\zeta|)^{(1+|\alpha|)/2}e^{-\frac c h \operatorname{Im}\zeta_1}e^{\frac1h|\operatorname{Im}\zeta'|}.
\end{align} 


Now we are ready to show the following proposition.  
\begin{proposition}\label{prop:Global result}
Let $\Omega\subset\R^n$, $n\geq 2$ be an open bounded connected set with smooth boundary and $\Gamma_2$ be an open nonempty proper subset of $\p\Omega$. Denote $\widetilde\Gamma_2=\p\Omega\backslash\Gamma_2$. Let $F\in L^\infty(\Omega;\C^n)$ and $g\in L^\infty(\Omega; \C)$. Suppose for all harmonic functions $v_{m+1}\in C^\infty(\overline{\Omega})$ with $\supp (v_{m+1}|_{\p\Omega})\subset \Gamma_2$ such that
\begin{align}\label{identity Fg}
    F(x)\cdot Dv_{m+1}(x)+g(x)v_{m+1}(x)=0 \qquad a.e. \textrm{ in }\Omega.
\end{align}  
Then we have $F=0$ and $g=0$ a.e. in $\Omega$.
\end{proposition}

\begin{remark}
We remark here that in the full data setting  $(\Gamma_2=\p\Omega)$, an individual proof of this proposition can be found in Appendix~\ref{append:alter_pf}. 
\end{remark}

\begin{proof}
\noindent\textbf{Step $1$: Local result.} 
As discussed above, we take $x_0=0\in \Gamma_2$ without loss of generality.
Substituting $v_{m+1}$ defined in \eqref{fun vm+1} with nontrivial boundary data $v_{m+1}|_{\p\Omega}\neq 0$ into \eqref{identity Fg}, we obtain
\begin{align}\label{local}
F(x)\cdot\zeta=hF(x)\cdot D\widetilde w e^{ix\cdot\zeta/h}+hg(x)+hg(x)\widetilde w e^{ix\cdot\zeta/h}.
\end{align} 
Let $\zeta=(i,1,0,\ldots,0)^T$. 
Then for all $x\in\overline{\Omega}$ such that $x_1>-c$, we have from the above estimates \eqref{diff w} for $\widetilde w$ that
\[|\widetilde w(x) e^{ix\cdot\zeta/h}|, \quad |D\widetilde w(x) e^{ix\cdot\zeta/h}|\leq Ce^{-\frac{x_1}h \operatorname{Im}\zeta_1}(1+h^{-1}|\zeta|)^{(1+|\alpha|)/2}e^{-\frac ch\operatorname{Im}\zeta_1}e^{\frac1h|\operatorname{Im}\zeta'|}.\]
When $h\rightarrow0$, this implies that the right-hand side of \eqref{local} $$hF(x)\cdot D\widetilde w e^{ix\cdot\zeta/h}+hg(x)+hg(x)\widetilde w e^{ix\cdot\zeta/h}$$ vanishes. Hence $F(x)\cdot\zeta=0$, at every point $x\in\overline\Omega$ with $x_1>-c$, i.e., in a neighborhood of $x_0=0$. Similarly, by choosing $\zeta'=(i,-1,0,\ldots,0)^T$ instead, we can derive that $F(x)\cdot \zeta'=0$. These two identities $F(x)\cdot\zeta=0$ and $F(x)\cdot\zeta'=0$ indicate the first two components of $F(x)$ indeed vanish. 

Furthermore, by choosing other 
\[\zeta=i{\bf e}_1 +{\bf e}_j\qquad \textrm{ for } j=3,\ldots, n,\]
one can show that the other components of $F(x)$ vanish too, which implies that $F=0$. 
Thus, we can also obtain that $g=0$ from the equation \eqref{identity Fg} and the fact that $v_{m+1}^{-1}(0)$ has measure zero. Finally we have derived that
$F=0$ and $g=0$ in a neighborhood of every point $x_0\in\Gamma_2$ provided that \eqref{identity Fg} hold for all harmonic functions $v_{m+1}$ with boundary data that is supported in $\Gamma_2$.  \\

\noindent\textbf{Step $2$: Global result.} To extend the local result to any point $x_1$ of $\Omega$, we take a point $x_0\in\Gamma_2$ and let $\theta: [0,1]\rightarrow\overline{\Omega}$ be a $C^1$ curve joining $x_0$ and $x_1$ such that $\theta(0)=x_0$ and $\theta'(0)$ is the inner normal to $\p\Omega$ at $x_0$, and $\theta(t)\in\Omega$ for $t\in(0,1]$. We set
\[\Uptheta_\varepsilon(t)=\{x\in\overline\Omega:~d(x, \theta([0,t]))\leq \varepsilon\},\]
a closed neighborhood of the curve $\theta(s)$, $s\in[0,t]$. Let 
\[I=\{t\in[0,1]:~F=0,\quad g=0\quad a.e. \;\textrm{on } \Uptheta_\varepsilon(t)\cap \Omega\}.\]
The above local result indicates that $0\in I$ if $\varepsilon>0$ is small enough. Moreover, it is clear that $I$ is a closed subset of $[0,1]$. If we can further show that $I$ is also open, then we can get $I=[0,1]$, which further implies that $x_1\notin \supp (F)\cup\supp(g)$. Since $x_1$ is an arbitrary point in $\Omega$, we then have $F=0$ and $g=0$ in $\Omega$. This will complete the proof of the global result.

To show that $I$ is open in $[0,1]$, we take $t\in I$ and $\varepsilon>0$ small enough so that $\p\Uptheta_\varepsilon(t)\cap\p\Omega\subset \Gamma_2$. It is easy to see that the set $\Omega\backslash \Uptheta_\varepsilon(t)$ can be smoothed out into an open subset $\Omega_1$ of $\Omega$ with smooth boundary so that
\[\Omega_1\supset\Omega\backslash\Uptheta_\varepsilon(t),\qquad \p\Omega\cap\p\Omega_1\supset\widetilde\Gamma_2.\]

We further augment the set $\Omega$ by smoothing out the set $\Omega\cup B(x_0,\varepsilon')$ with $0<\varepsilon'\ll \varepsilon$ sufficiently small, into an open set $\Omega_2$ so that 
$$
\p\Omega_2\cap\p\Omega\supset \p\Omega\cap\p\Omega_1=\p\Omega_1\cap\p\Omega_2\supset \widetilde{\Gamma}_2.
$$

Now we let $G_2$ be the Green kernel associated to the open set $\Omega_2$ and 
$$
-\Delta_y G_2(x,y) = \delta(x-y)\quad \hbox{ in }\Omega_2,\ \ G_2(x,y)|_{\p\Omega_2} =0.
$$
We consider the function 
\[\Phi(x;y):=F(y)\cdot D_yG_2(x,y)+g(y)G_2(x,y),\quad y\in\Omega_1,\quad x\in\Omega_2\backslash\overline{\Omega_1}.\]
It is clear that $\Phi(x;y)$ is harmonic in $x$ on $\Omega_2\backslash\overline{\Omega_1}$ for a fixed $y\in \Omega_1$. Since $F(y)=0$ and $g(y)=0$ for $y\in\Uptheta_\varepsilon(t)\cap\Omega$, we can extend $\Phi(x;y)$ by zero to $y\in\Omega$. When $x\in \Omega_2\backslash\overline{\Omega}$, the Green function $G_2(x,y)$ is a harmonic function in $y$ on $\Omega$ with $G_2(x,\cdot)|_{\widetilde\Gamma_2}=0$. By \eqref{identity Fg}, we have
\[\Phi(x;y)=0,\qquad \textrm{ for } a.e. \quad y\in \Omega, \quad x\in \Omega_2\backslash\overline{\Omega}.\]
Since $\Phi(x;y)$ is harmonic in $x$ on $\Omega_2\backslash\overline{\Omega_1}$ and the set $\Omega_2\setminus\overline{\Omega}_1$ is connected, by the unique continuation, we then have 
\[\Phi(x;y)=0, \qquad \textrm{ for } a.e. \quad y\in \Omega_1,\quad x\in \Omega_2\backslash\overline{\Omega_1}.\]

By Lemma 2.2 of \cite{KU201909} ($H^1$-density), we have that for any $v\in C^\infty(\overline{\Omega_1})$ harmonic with $v|_{\p\Omega_1\cap\p\Omega_2}=0$ and arbitrary small $\epsilon>0$, there exists $a\in C^\infty(\overline{\Omega_2})$ with $\supp(a)\subset \Omega_2\backslash\overline{\Omega_1}$ such that 
\[\left\|v(y)-\int_{\Omega_2}G_2(x,y)a(x)~dx\right\|_{H^1(\Omega_1)}<\epsilon.\]
We multiply $\Phi(x;y)$ by $a(x)$ and then integrate it with respect to $x$ on $\Omega_2$. We obtain
\[F(y)\cdot D_y\int_{\Omega_2}G_2(x,y)a(x)~dx+g(y)\int_{\Omega_2}G_2(x,y)a(x)~dx=0,\quad a.e. \quad y\in\Omega_1,\]
and, moreover, we can derive that
\begin{align*}
&\|F\cdot Dv+gv\|_{L^2(\Omega_1)}\\
&\leq \left\|F\cdot D\int_{\Omega_2}G_2(x,\cdot)a(x)~dx+g\int_{\Omega_2}G_2(x,\cdot)a(x)~dx\right\|_{L^2(\Omega_1)}+C\epsilon=C\epsilon
\end{align*}
for arbitrary small $\varepsilon>0$. This implies that 
\[F\cdot Dv+gv=0\qquad a.e. \textrm{ in }\Omega_1\]
for $v\in C^\infty(\overline{\Omega_1})$ harmonic with $v|_{\p\Omega_1\cap\p\Omega_2}=0$. By the above local result in Step $1$, we then have $F=0$ and $g=0$ in an open neighborhood of $\p\Omega_1\setminus (\p\Omega_1\cap\p\Omega_2)$ and this implies that $F$ and $g$ vanish on a slightly larger neighborhood $\Theta_\varepsilon(t'),\ t'>t$ of the curve. 
This proves that $I$ is open, hence completes the proof. 

\end{proof}

\begin{proof}[Proof of Theorem~\ref{thm:global uniqueness}]
From Proposition~\ref{HigherIntegralIdentity}, we have the integral identity holds for $m=2$. By applying Lemma~\ref{lemma transport}, Proposition~\ref{prop:Global result} and \eqref{Fg}, we have  
$ 
 F=0, \ g=0,
$  
which implies that $$\p_z A_1(x,0)= \p_zA_2(x,0),\qquad \p_z^2 q_1(x,0)=\p_z^2 q_2(x,0).$$
Given any integer $m>2$, by induction argument, suppose that for $k=2,\ldots,m-1$, the following are true:
$$\p_z^{k-1} A_1(x,0)= \p_z^{k-1} A_2(x,0),\qquad \p_z^k q_2(x,0)=\p_z^k q_1(x,0).$$ We want to show that $\p_z^{m-1} A_1(x,0)= \p_z^{m-1} A_2(x,0)$ and $\p_z^m q_1(x,0)=\p_z^m q_2(x,0)$ also hold.

From above, we have known that $A_j$ and $q_j$ satisfy the conditions \eqref{AAssumption} and \eqref{QAssumption} and thus we can apply Proposition~\ref{HigherIntegralIdentity} to get the integral \eqref{keyId} for such $m>2$. Applying Lemma~\ref{lemma transport} and Proposition~\ref{prop:Global result} again, we then derive that $F=0,\ g=0$, which gives
$$
    0= \p_z^{m-1}A_2(x,0) - \p_z^{m-1}A_1(x,0) 
$$
and 
$$
    0= \p_z^m q_2(x,0)-\p_z^m q_1(x,0).
$$
Therefore, we complete the proof of Theorem~\ref{thm:global uniqueness}.
\end{proof}

\appendix
\section{Well-posedness of the nonlinear magnetic Schr\"odinger equation}\label{sec:pre}

In this section, we prove that the boundary value problem \eqref{eqn:MagnSch} is well-posed if the small boundary data is given. The analysis is based on the contraction mapping principle.

\begin{theorem}[Well-posedness]\label{Thm:well-posedness}
	Let $A(x,z)$ and $q(x,z)$ satisfy \eqref{A_holomorphic}-\eqref{q_holomorphic}.  
	Moreover, suppose that $q(x,0)=0$ and $0$ is not a Dirichlet eigenvalue of the linear operator $$\mathcal{L}_0:=(D+A(x,0))^2+\p_z q(x,0).$$ Then there exists a small constant $\varepsilon>0$ such that for any $\|f\|_{W^{2-1/p,p}(\p\Omega)}\leq \varepsilon$, the boundary value problem
	\begin{align}
\left\{\begin{array}{ll}
\LC D + A (x,u) \RC^2 u + q(x,u) =0 & \hbox{in }\Omega,\\
u=f & \hbox{on }\p\Omega,\\
\end{array}\right.
\end{align}
	admits a unique solution $u\in W^{2,p}(\Omega)$. 
		Moreover, there exists a constant $C>0$ independent of $f$ such that 
		\begin{align}
		\|u\|_{W^{2,p}(\Omega)} \leq C\|f\|_{W^{2-1/p,p}(\p\Omega)}.
		\end{align}
\end{theorem}
\begin{proof}
	We will use contraction mapping principle to show the existence of solution to \eqref{eqn:MagnSch}. \\
	
\noindent\textbf{Step 1: Linearization.} 	{First, for $A(x,z)$ and $q(x,z)$ satisfying \eqref{A_holomorphic}-\eqref{q_holomorphic}, we use the Taylor formulas
		\[\begin{split}A(x,z)& 
		=A(x,0)+A_r(x,z)z,\\
		q(x,z)& 
		=\partial_zq(x,0)z+q_r(x,z)z^2,\end{split}\]
where we denote
$$
    A_r(x,z):=\int_0^1\partial_zA(x,tz)~dt,\qquad q_r(x,z):=\int_0^1\partial_z^2q(x,tz)(1-t)~dt.
$$	
	
		Given $f\in W^{2-1/p, p}(\partial\Omega)$ for $p\in (n,+\infty)$, by Theorem 9.15 of \cite{gilbarg2015elliptic}, there exists a unique solution $u_0\in W^{2,p}(\Omega)$ to the Dirichlet problem  
		\begin{align}\label{eqn: u0}
		\left\{\begin{array}{ll}
		\mathcal{L}_{0}u_0:=(D+A(x,0))^2 u_0 + \p_z q(x,0)u_0 = 0 & \hbox{in }\Omega,\\
		u_0 = f & \hbox{on }\p\Omega.\\
		\end{array}\right.
		\end{align}
		Moreover, we have 
		\[\|u_0\|_{W^{2,p}(\Omega)}\leq C\|f\|_{W^{2-1/p,p}(\partial\Omega)}.\]
		(This can be obtained by extending $f$ to a $W^{2,p}(\Omega)$ function and apply Lemma 9.17 of \cite{gilbarg2015elliptic} to the equation for the difference of the solution and the extended function.)
		
        Thus, if $u$ is a solution to \eqref{eqn:MagnSch} we have the remainder function $v:=u-u_0$ satisfying the following problem	
		\begin{equation}\label{eqn:v}\mathcal{L}_{0} v=\mathcal F(v),\qquad v|_{\partial\Omega}=0,\end{equation}
		where  
		\[\begin{split}\mathcal F(v):=&-\left(D+A(x,0)\right)\cdot\left[A_r(x,u_0+v)(u_0+v)^2\right]\\
		&-A_r(x,u_0+v)(u_0+v)\cdot \left(D+A(x,0)\right)(u_0+v)\\
		&-A_r(x,u_0+v)^2(u_0+v)^3-q_r(x,u_0+v)(u_0+v)^2.
		\end{split}\]
		By Theorem 9.15 in \cite{gilbarg2015elliptic} again, for $F\in L^p(\Omega)$, there exists a unique solution $\tilde u\in W^{2,p}(\Omega)\cap W^{1,p}_0(\Omega)$ to the equation $\mathcal{L}_{0} \tilde u=F\in L^p(\Omega)$ in $\Omega$ with trivial boundary data. We denote the solution operator  by 
		$$
		\mathcal{L}_{0}^{-1}:~L^p(\Omega)~\rightarrow~W^{2,p}(\Omega)\cap W^{1,p}_0(\Omega),
		$$ 
		which is the continuous operator $F\mapsto \tilde u$ and thus $\mathcal{L}_{0}^{-1}(F)$ is the solution to $\mathcal{L}_{0} \tilde u=F\in L^p(\Omega)$ in $\Omega$ with trivial boundary condition. Therefore, we are looking for the unique fixed point $v$ of $\mathcal L_0^{-1}\circ\mathcal F$.\\

\noindent\textbf{Step 2: A contraction map.} 
		In what follows, we will show that $\mathcal{L}_{0}^{-1}\circ \mathcal F$ is indeed a contraction map on a suitable subset $X_\delta$ of $W^{2,p}(\Omega)\cap W^{1,p}_0(\Omega)$. 
		Here we denote the set $X_\delta$ for $1>\delta>0$ by
		$$X_\delta:=\{v\in W^{2,p}(\Omega)\cap W^{1,p}_0(\Omega)~|~\|v\|_{W^{2,p}(\Omega)}\leq \delta\}.$$

		We first show that $(\mathcal{L}_{0}^{-1}\circ \mathcal F)(X_\delta)\subset X_\delta$.
		Recalling that, by the Sobolev embedding theorem, we have $W^{2,p}(\Omega)\hookrightarrow C^1(\Omega)$ if $p>n$. 
		For $v\in X_\delta$, we have $v, u_0\in C^1(\Omega)$ since $v, u_0\in W^{2,p}(\Omega)$. Thus we have that $A_r(x, u_0(x)+v(x))$ and $q_r(x, u_0(x)+v(x))$ are both bounded in $\Omega$. Moreover, since 
		\begin{align}\label{eqn:exp_DA}
		&D\cdot [A_r(x, u_0(x)+v(x))] \notag \\
		&=\int_0^1D_x\cdot \p_zA(x, t(u_0+v))~dt+\int_0^1 t \p_z^2A(x,t(u_0+v))~dt\cdot D(u_0+v),\end{align}
	    one can derive that 
		\begin{align*}
		\|\mathcal{F}(v)\|_{L^p(\Omega)}&\leq C\|u_0+v\|_{C^1(\Omega)}\|u_0+v\|_{L^p(\Omega)} \\
		&\leq C\|u_0+v\|_{W^{2,p}(\Omega)}^2\leq C(\|u_0\|_{W^{2,p}(\Omega)}^2+\|v\|_{W^{2,p}(\Omega)}^2).
		\end{align*}
		This implies that, for $\|f\|_{W^{2-1/p,p}(\p\Omega)}<\varepsilon$ and $p>n$, one has
		\begin{align}\label{eqn:est_LF}
		\|\mathcal{L}_{0}^{-1}(\mathcal F(v))\|_{W^{2,p}(\Omega)}
		&\leq C\|\mathcal F(v)\|_{L^p(\Omega)} \notag\\		
		&\leq C(\|f\|_{W^{2-1/p,p}(\p\Omega)}^2+\|v\|_{W^{2,p}(\Omega)}^2)\leq C(\varepsilon^2+\delta^2).\end{align}
		Therefore, for $\varepsilon$ and $\delta$ small enough, the operator $L_{0}^{-1}\circ \mathcal F$ maps $X_\delta$ into itself. 
		
		Next we show that $\mathcal{L}_{0}^{-1}\circ\mathcal F$ is a contraction on $X_\delta$. To this end, we take $v_1, v_2\in X_\delta$ and consider 
		\[\begin{split}\|\mathcal{L}_{0}^{-1}\circ\mathcal F(v_1)-\mathcal{L}_{0}^{-1}\circ\mathcal F(v_2)\|_{W^{2,p}(\Omega)}&=\|\mathcal{L}_{0}^{-1}(\mathcal F(v_1)-\mathcal F(v_2))\|_{W^{2,p}(\Omega)}\\
		&\leq C\|\mathcal F(v_1)-\mathcal F(v_2)\|_{L^p(\Omega)}.
		\end{split}\]
		In addition, we rewrite 
		\[\begin{split}-\mathcal F(v)&=D\cdot (A_r(x,u_0+v))(u_0+v)^2+3(A_r(x,u_0+v)\cdot D(u_0+v))(u_0+v)\\
		&\quad+2A(x,0)\cdot A_r(x,u_0+v)(u_0+v)^2+A_r(x,u_0+v)^2(u_0+v)^3\\
		&\quad+q_r(x,u_0+v)(u_0+v)^2.
		\end{split}\]
		Thus, $\mathcal F(v_2)-\mathcal F(v_1)$ is the sum of the following two terms
		\[\begin{split}
		\textrm{I} &= D\cdot (A_r(x, u_0+v_1))[(u_0+v_1)^2-(u_0+v_2)^2]\\
		&\quad+3A_r(x,u_0+v_1)\cdot [D(u_0+v_1)(u_0+v_1)-D(u_0+v_2)(u_0+v_2)]\\
		&\quad+2A(x,0)\cdot A_r(x,u_0+v_1)[(u_0+v_1)^2-(u_0+v_2)^2]\\
		&\quad+A_r(x,u_0+v_1)^2[(u_0+v_1)^3-(u_0+v_2)^3]\\
		&\quad+q_r(x,u_0+v_1)[(u_0+v_1)^2-(u_0+v_2)^2],\\
		\textrm{II} &= [D\cdot (A_r(x,u_0+v_1))-D\cdot (A_r(x,u_0+v_2))](u_0+v_2)^2\\
		&\quad+3(A_r(x,u_0+v_1)-A_r(x,u_0+v_2))\cdot D(u_0+v_2)(u_0+v_2)\\
		&\quad+2A(x,0)\cdot[A_r(x,u_0+v_1)-A_r(x,u_0+v_2)](u_0+v_2)^2\\
		&\quad+[A_r(x, u_0+v_1)^2-A_r(x, u_0+v_2)^2](u_0+v_2)^3\\
		&\quad+(q_r(x,u_0+v_1)-q_r(x,u_0+v_2))(u_0+v_2)^2.
		\end{split}\]
		For the first term, we obtain
		\[\begin{split}\|\textrm I\|_{L^p(\Omega)} &\leq  C\big\{(\|u_0\|_{C^1(\Omega)}+\|v_1\|_{C^1(\Omega)}+\|v_2\|_{C^1(\Omega)})\|v_1-v_2\|_{L^p(\Omega)}\\
		&\quad+(\|u_0\|_{L^p(\Omega)}+\|v_1\|_{L^p(\Omega)}+\|v_2\|_{L^p(\Omega)} )\|v_1-v_2\|_{C^1(\Omega)}\big\}\\
		&\leq C(\|u_0\|_{W^{2,p}(\Omega)}+\|v_1\|_{W^{2,p}(\Omega)}+\|v_2\|_{W^{2,p}(\Omega)})\|v_1-v_2\|_{W^{2,p}(\Omega)}\\
		&\leq C(\varepsilon+\delta)\|v_1-v_2\|_{W^{2,p}(\Omega)}.
		\end{split}\]
		For II, we have
		\[\begin{split}
		&\|\textrm{II}\|_{L^p(\Omega)}\\
		&\leq C\|u_0+v_2\|_{C^1(\Omega)}^2\big\{\|D\cdot [A_r(x,u_0+v_1)]-D\cdot [A_r(x, u_0+v_2)]\|_{L^p(\Omega)}\\
		&\quad +\|A_r(x,u_0+v_1)-A_r(x,u_0+v_2)\|_{L^p(\Omega)}+\|q_r(x,u_0+v_1)-q_r(x,u_0+v_2)\|_{L^p(\Omega)}\big\}.
		\end{split}\]
		By \eqref{eqn:exp_DA} and that $D\cdot\partial_zA(x,z)$, $\partial_z^2A(x,z)$, $\partial_zA(x,z)$ and $\partial_z^2q(x,z)$ are all Lipschitz in $z$ {(where the Lipschitz constants are independent of $x$ by the boundedness of $\p_z^k A$ and $\p_z^k q$)}, we obtain
		\[\|\textrm{II}\|_{L^p(\Omega)}\leq C\|u_0+v_2\|^2_{C^1(\Omega)}\|v_1-v_2\|_{W^{1,p}(\Omega)}\leq C(\varepsilon^2+\delta^2)\|v_1-v_2\|_{W^{2,p}(\Omega)}.
		\] 
		Combining above estimates together, we obtain 
		\[\|\mathcal F(v_1)-\mathcal F(v_2)\|_{L^p(\Omega)}\leq C(\delta+\varepsilon+\delta^2+\varepsilon^2)\|v_1-v_2\|_{W^{2,p}(\Omega)}.\]
		Therefore, $\mathcal{L}_{0}^{-1}\circ\mathcal F$ is a contraction on $X_\delta$ for $\varepsilon$ and $\delta$ small enough. Using the contraction mapping theorem, there exists a unique fixed point $v\in X_\delta$ of $\mathcal{L}_{0}^{-1}\circ \mathcal F$, namely, 
		$$(\mathcal{L}_{0}^{-1}\circ \mathcal F)(v)=v,$$
        and hence $v$ solves \eqref{eqn:v}. Substituting the fixed point $v$ into the second inequality of \eqref{eqn:est_LF}, we then have
		\[\|v\|_{W^{2,p}(\Omega)}\leq C(\varepsilon\|f\|_{W^{2-1/p,p}(\partial\Omega)}+\delta\|v\|_{W^{2,p}(\Omega)}).\]
		For $\delta$ small enough, this gives
		\[\|v\|_{W^{2,p}(\Omega)}\leq C\|f\|_{W^{2-1/p,p}(\partial\Omega)}.\]
		Finally, we obtain $u=u_0+v\in W^{2,p}(\Omega)$ which solves \eqref{eqn:MagnSch} and satisfies
		\[\|u\|_{W^{2,p}(\Omega)}\leq C\|f\|_{W^{2-1/p,p}(\partial\Omega)}.\]

	}
 
\end{proof}

\section{An alternative proof of the full boundary data result}\label{append:alter_pf}
In this section, we provide a separate proof to show that the nonlinear potentials can be uniquely recovered when the boundary data are given on the whole boundary.

\begin{proof}[Proof of Theorem~\ref{thm:global uniqueness} \textup{(}when $\Gamma_1=\Gamma_2=\p\Omega)$]
We will begin by reproving Proposition~\ref{prop:Global result} here when $\Gamma_2=\p\Omega$. From identity \eqref{Fg_ID}, we substitute harmonic function
\[ 
v_{m+1} = e^{\zeta  \cdot x},
\]
into \eqref{Fg_ID}, where $\zeta \in \mathbb{C}^n$ satisfy $\zeta \cdot\zeta =0.$
Then we have
\begin{align}\label{A_ID0}
    F(x)\cdot\zeta +g(x)=0.
\end{align}
Since $\zeta$ is arbitrary with $\zeta\cdot\zeta=0$, we can take 
$$
\zeta=h{\bf e}_1 +i h{\bf e}_j  
$$
for $j=2,\ldots,n$ and $h\in\R$.
We then obtain from \eqref{A_ID0} that
\begin{align}\label{A_ID1}
     F(x)\cdot ({\bf e}_1+i {\bf e}_j  )= 0
\end{align}
as $h\rightarrow \infty$.
Similarly, we can take
$$
\zeta'=h{\bf e}_1-i h{\bf e}_j ,
$$
then we have
\begin{align}\label{A_ID2}
F(x)\cdot ({\bf e}_1-i {\bf e}_j )= 0.
\end{align}
Adding these two equations \eqref{A_ID1} and $\eqref{A_ID2}$ together, we get
$$
   F(x)\cdot {\bf e}_1=0,
$$ 
which implies the first component of $F$ vanishes. Following similar argument as above, we can then conclude $F=0$ in $\Omega$. Thus, from \eqref{Fg_ID}, we can also derive $g=0$ if we have known $F=0$. 

Finally, by following a similar argument as in the Proof of Theorem~\ref{thm:global uniqueness} for the partial data setting in Section~\ref{sec:partial data}, we obtain the uniqueness result with complete data.
\end{proof}

\vskip1cm

\bibliographystyle{abbrv}
\bibliography{NLSref}

\end{document}